\documentclass{article}
\usepackage{amssymb}
\usepackage{amsthm}
\usepackage{amsmath}
\usepackage{enumerate}
\usepackage{setspace}
\usepackage{bm}
\usepackage{mathrsfs}
\usepackage{graphicx}
\usepackage{color}
\usepackage[hypertex]{hyperref}
\usepackage[all]{xy}
\theoremstyle{definition}
\newtheorem{thm}{Theorem}
\newtheorem*{thm*}{Theorem}
\newtheorem{lem}[thm]{Lemma}
\newtheorem{prop}[thm]{Proposition}
\newtheorem{dfn}[thm]{Definition}

\newtheorem{rmk}[thm]{Remark}

\newcommand{\T}{\mathbb{T}}
\newcommand{\R}{\mathbb{R}}

\newcommand{\Z}{\mathbb{Z}}

\newcommand{\image}{\operatorname{Im}}

\newcommand{\supp}{\operatorname{supp}}
\newcommand{\divisor}{\operatorname{div}}

\begin{document}

\title{Riemann-Roch inequality for smooth tropical toric surfaces}


\author{Ken Sumi}
\date{}


\maketitle

\begin{abstract}
For a divisor $D$ on a tropical variety $X$,
we define two amounts  in order to estimate the value of $h^{0}(X,D)$,
which are described by terms of global sections and
computed more easily than $h^{0}(X,D)$.
As an application of its estimation, we show that a Riemann-Roch inequality holds for
smooth tropical toric surfaces.
\end{abstract}

\section{Introduction}



In this paper, we prove a Riemann-Roch inequality for smooth tropical toric surfaces
by calculating $h^{0}(X,D)$, which is a higher dimensional generalization of the rank of divisor.

For a divisor $D$,
the rank of divisor $r(D)$ on a finite graph is defined by Baker-Norine\cite{BN}.
Gathmann-Kerber\cite{GK} and Mikhakin-Zharkov\cite{MZ}
generalized the definition of the rank of divisor  for tropical curves and
showed a Riemann-Roch theorem for compact tropical curves.

The rank of divisor is defined as the maximum number $k$ such that $|D-E|\neq \emptyset$
for arbitrary effective divisor $E$ of degree $k$.
It is an alternative notion of the dimension of the projectivization of $H^{0}(X,\mathcal{O}(D))$,
however, the rank of divisor 
$r(D)$ cannot be given only by the correspondence tropical module $\Gamma(X,\mathcal{O}(D))$.
For example,
\cite{yoshitomi}[Example 6.5] says that there are two pairs $(C,D)$ and $(C',D')$
of tropical curves and divisors such that spaces of global sections
$H^{0}(C,\mathcal{O}(D))$,
$H^{0}(C',\mathcal{O}(D'))$
are isomorphic, while $r(D)\neq r(D')$.

A higher dimensional generalization of $r(D)$ (precisely, $r(D)+1$) is introduced in Cartwright\cite{DC2},
denoted by $h^{0}(X,D)$, and formulate a Riemann-Roch inequality for tropical surfaces.
The author computed the value of $h^{0}(X,D)$ for a divisor $D$ on a tropical torus $X$
and showed that the Riemann-Roch inequality formulated by Cartwright
 holds for tropical Abelian surfaces in \cite{ken}.


In this paper, we generalize the method of the computation of the value of $h^{0}(X,D)$ used in
the author's previous work \cite{ken} for general tropical varieties and divisors.
Concretely, we define two amounts $h_{a}^{0}$ and $h_{b}^{0}$.
We define $h_{a}^{0}$ as the maximal number $k$ such that
there is a tropical submodule of $\Gamma(X,\mathcal{O}(D))$
satisfying that it is generated by $k$ elements and its generators satisfies
the indentity theorem, that is, two generators are distinct if and only if
these are distinct on any open subset of $X$.
In general, the identity theorem does not hold in tropical geometry since
tropical regular function means a piecewise-linear convex function with integral slopes.
We define $h_{b}^{0}$ as the maximal number $k$ such that
for all $x\in X$ there exist $k$ sections satisfying that 
each two sections of these $k$ sections are distinct `around' $x$.

Then we get the following inequality.
\begin{prop}[Proposition \ref{Mainresult1}]
\[ h^{0}_{a}(X,D)\leq h^{0}(X,D)\leq h^{0}_{b}(X,D)\]
\end{prop}
When we have a concrete minimal generating set of $\Gamma(X,\mathcal{O}(D))$,
we can compute these amounts $h_{a}^{0}$ and $h_{b}^{0}$ more easily than computing $h^{0}(X,D)$
since $h^{0}(X,D)$ is defined by terms of divisors.

Tropical toric varieties are studied by many mathematicians.
For example,
Mayer\cite{HM} and Shaw \cite{S0} studied intersection theory on tropical toric varieties, especially
smooth tropical toric surfaces.

The author shows the following basic statement.
\begin{prop}[Proposition \ref{gens}]
Let $X$ be a smooth tropical toric variety and $D$ be a divisor on $X$. Then
\[ \Gamma (X,\mathcal{O}(D))\cong ``\sum_{\divisor (x^{m})+D\geq 0 } \T \cdot x^{m}" . \]
\end{prop}
By using the above two propositions and classical Riemann-Roch inequality for smooth tropical toric
surfaces, we get the value of $h^{0}(X,D)$ and show
the Riemann-Roch inequality for smooth tropical toric surfaces.

\section{Preliminaries}

\subsection{Tropical modules }

The tropical semifield $\T$ is the set $\R\cup \{ -\infty \}$ equipped with
the tropical sum $``x+y"=\max \{ x,y \}$ and
the tropical product $``xy"=x+y$ for $x,y\in \T$.
A commutative semigroup $V$  
equipped with the unit $\{ -\infty\}$ and a scalar product $\T\times V \rightarrow V$ 
is called a {\it tropical module} or a $\T${\it -module}
if the following conditions are satisfied:
\begin{itemize}
\item $``(x+y)v"=``xv+yv"$ for any $x,y\in \T$, any $v\in V$;
\item $``x(v+w)"=``xv+xw"$ for any $x \in \mathbb{T}$, any $v,w\in V$;
\item $``x(yv)"=``(xy)v"$ for any $x,y\in \mathbb{T}$, any $v\in V$;
\item $``0\cdot v"=v$ for the multiplicative unit $0\in \T$ and any $v\in V$;
\item if $``xv"=``yv"$ for some $x,y \in \T$ and $v\in \T$,
then $x=y$ or $v=-\infty$.
\end{itemize}

Let $V$, $W$ be tropical modules.
A map $f\colon V\to W$ is called a {\it tropical linear morphism} if
$f$ satisfies $f(``v_{1}+v_{2}")=``f(v_{1})+f(v_{2})"$ and $f(``tv")=``tf(v)"$ for any $v,v_{1},v_{2}\in V$, $t\in \T$.
If there is the set theoretical inverse map $f^{-1}$ and it is also a tropical linear morphism,
then we call $f$ an isomorphism.


A function $f:\R ^{n}\rightarrow \T$ is a {\it tropical polynomial} (resp. {\it tropical Laurent polynomial})
if $f$ is a constant map to $\{ -\infty\}$ or is
of the form 
$f(x)=\displaystyle\max_{j\in S}\{ a_{j}+j\cdot x \}$,
where $S$ is a finite subset of $(\Z_{\geq 0})^{n}$ (resp. $\Z^{n}$) and 
$a_{j}\in \R$. 
A tropical polynomial naturally extends to a function from $\T^{n}$ to $\T$.
The set of tropical polynomials on $\R^{n}$ is naturally a tropical module.

A {\it $\Z$-affine linear function} is a tropical Laurent monomial, that is, 
an affine linear function on $\R^{n}$ whose slope is in $(\Z ^{n})^{*}$.

For a matrix $(t_{ij})_{1\leq i,j \leq l}$ with $t_{ij}\in \T$,
we define the tropical determinant $\operatorname{det}_{\text{trop}}(t_{i,j})$ to be 
\[ \max_{\sigma\in \mathfrak{S}_{l}}\left\{ \sum_{i=1}^{n}t_{\sigma (i)i} \right\} ,\] 
where $\mathfrak{S}_{l}$ is the symmetric group of $\{ 1,\ldots ,l \}$.

\begin{dfn}
Let $M$ be a tropical module.
An element $m\in M$ is {\it an extremal} if there exist no elements $a,b\in M$
such that $ ``a+ b" =m$ and $a,b\neq m$
\end{dfn}


\subsection{Tropical varieties}

A tropical variety is defined as a balanced polyhedral complex with weight $1$
whose each cells are locally a convex rational polyhedron in $\T^{n}$.

\begin{dfn}
The {\it sedentarity}  of a point $t=(t_{1},\ldots ,t_{n})\in \T^{n}$
 is the set $I:=\{ i\in \{ 1,\ldots ,n \} \mid  t_{i}=-\infty \}$.
A {\it face} of $\T^{n}$ of sedentarity $I$ is a set
$\{  t\in \T^{n} \mid t_{i}=\infty \text{ for all }i\in I \}$.
The face of sedentarity $\emptyset$ is $\T^{n}$ itself.
\end{dfn}

\begin{dfn}
An {\it rational polyhedron $P$ in $\T^{n}$} is a finite intersection of
a face of $\T^{n}$  and half-spaces
defined by finite inequalities $a_{i}\cdot x +b_{i}\geq 0, i\in I$
 where
$x\in \R^{n}$, $a\in \Z^{n}$, and $b\in \R$.
A {\it face of $P$} is an intersection of $P$ and a face of $\T^{n}$ or
a hyperplane $a\cdot x +b= 0$, $a\in \Z^{n}$, $b\in \R$
appearing the definition of $P$.


A face or a point is {\it inner} if it is contained in $\R^{n}$.
\end{dfn}

\begin{dfn}
An {\it rational polyhedral complex in $\T^{n}$} is an union of
a collection of rational polyhedra in $\T^{n}$ such that
the intersection of any subcollection is a common face of
all rational polyhedra in the subcollection.

A {\it cell of rational polyhedral complex} $\mathcal{P}$ is
a face of an rational polyhedron in the collection.
A {\it facet of $\mathcal{P}$} is face which is not contained in any other face.

The {\it dimension of} an rational polyhedral complex is the maximal dimension
of rational polyhedra in the collection.
If all facets of an rational polyhedral complex is $n$-dimensional, then
the complex is said to be pure $n$-dimensional or pure of dimension $n$.

An {\it $n$-dimensional weighted rational polyhedral complex}
\ is a pure $n$-dimensional rational polyhedral complex with a weight function
which is a $\Z$-valued function on the set of $n$-dimensional cells.
Then we can consider an rational polyhedral complex as
a weighted complex whose weight function is the constant $1$.
\end{dfn}

\begin{dfn}\label{bal}
Let $P$ be a $d$-dimensional weighted rational polyhedral complex in $\T^{n}$
with a weight function $w$ and $E$ be a $(d-1)$-dimensional cell of $P$ with $E\cap \R^{n}\neq \emptyset$.
Let $F_{1},\ldots ,F_{k}$ be the $d$-dimensional cells adjacent to $E$ and
$\pi \colon \R^{n}\to \R^{n-d+1}$ be the projection whose kernel is parallel to $E$.
Then $\pi (F_{1}),\ldots ,\pi (F_{k})$ are rays adjacent to the point $\pi (E)$.
Let $v_{1},\ldots ,v_{k}\in \Z^{n-d+1}$ be the primitive vectors parallel to
$\pi (F_{1}),\ldots ,\pi (F_{k})$ respectively.
Then $P$ is {\it balanced at} $E$ if 
$\sum_{i=1,\ldots ,k}w(F_{i})v_{i}=1$ and
 $P$ is {\it balanced} if $P$ is balanced at $E$ for all 
$(d-1)$-dimensional faces $E$ of $P$.
\end{dfn}




\begin{dfn}
A {\it tropical variety} is
a topological space $X$ with an atlas 
$\{ (U_{\alpha }, \Phi _{\alpha} )\}_{\alpha}$ such that
\begin{itemize}
\item $U_{\alpha }$ is homeomorphic to an open subset of 
an $n$-dimensional balanced rational polyhedral complex with weight $1$ in $\T ^{n_{\alpha}}$
by  $\Phi _{\alpha}:U_{\alpha }\rightarrow \T ^{n_{\alpha}}$, 
which is an isomorphism to its image $\Phi (U_{\alpha})$;
\item $\Phi_{\alpha }(U_{\alpha})\cap \R^{n_{\alpha}}\neq \emptyset$;
\item the transition function on the ``finite part''
$\iota _{n_{\alpha}}^{-1}\circ \Phi _{\alpha}\circ
\Phi _{\beta}^{-1}\circ \iota _{n_{\beta}}$ 
is a restriction of a $\Z$-affine linear map
from $\R^{n_{\beta}}$ to $\R ^{n_{\alpha}}$.
Here $\iota_{n}\colon \R^{n}\to \T^{n}$ is the natural embedding.
\end{itemize}

If every $U_{\alpha}$ is homeomorphic to an open subset of
an $n$-dimensional balanced rational polyhedral complex,
a tropical variety $X$ is called {\it $n$-dimensional}.
Especially, $1$-dimensional tropical varieties are called {\it tropical curves}
and $2$-dimensional tropical varieties are called {\it tropical surfaces}.
\end{dfn}

\begin{dfn}
Let $U$ be an open subset of a tropical variety.
A continuous function $f:U\to \T\cup \{ -\infty \}$ is said to be {\it rational } if
there exist charts $\{ (U_{\alpha_{i}},\Phi_{\alpha_{i}})  \}_{i}$
and tropical polynomials $g_{\alpha_{i}},h_{\alpha_{i}}$ on $\T^{\alpha_{i}}$
satisfying that 
$U=\bigcup U_{\alpha_{i}}$ and $f\circ \Phi_{\alpha_{i}}^{-1}=g_{\alpha_{i}}-h_{\alpha_{i}}$ on 
$\Phi_{\alpha_{i}}(U_{\alpha_{i}})$.

It is said to be {\it regular } if we can take a $\Z$-affine linear function as $h_{\alpha_{i}}$
in the above notation.
\end{dfn}

\begin{dfn}
The sedentarity number of a point $p$ of a tropical variety $X$ is the minimal cardinal of the
sedentarity of $\Phi_{\alpha}(p)\in U_{\alpha}$ when $\alpha$ runs..

A point $p$ of a tropical variety $X$ is said to be {\it inner}
if the sedentarity number of $p$ is zero.
We denote the set of inner points of $X$ by $X^{\circ}$.
\end{dfn}

\subsection{Cartier divisors and line bundles}


\begin{dfn}
A (Cartier) divisor on $X$ is a collection of pairs
$D=\{  (U_{\alpha}, f_{\alpha}) \}_{\alpha}$ of an open subset $U_{\alpha}$ and a rational function $f_{\alpha}$ on $U_{\alpha}$ satisfying that $\bigcup_{\alpha}U_{\alpha}=X$
for any two charts $U_{\alpha}$ and $U_{\beta}$ with $U_{\alpha}\cap U_{\beta}\neq \emptyset$
there is an invertible tropical function $d_{\alpha ,\beta}$ on $U_{\alpha}\cap U_{\beta}$ which coincides with
$f_{\alpha}-f_{\beta}$ at any inner point of $U_{\alpha}\cap U_{\beta}$.
We identify
two divisors $\{  (U_{\alpha}, f_{\alpha}) \}_{\alpha}$, $\{  (V_{\beta}, g_{\beta}) \}_{\beta}$
if $\{  (U_{\alpha}, f_{\alpha}) \}_{\alpha}\cup\{  (V_{\beta}, g_{\beta}) \}_{\beta}$ is also a divisor.
\end{dfn}

We can define the sum of two divisors 
$D=\{  (U_{\alpha}, f_{\alpha}) \}_{\alpha}$ and $D'=\{  (V_{\beta}, g_{\beta}) \}_{\beta}$
to be a divisor 
$D+D':=\{  (U_{\alpha}\cap V_{\beta}, f_{\alpha}+ g_{\beta}) \}_{\alpha,\beta}$.

For a divisor $D=\{  (U_{\alpha}, f_{\alpha}) \}_{\alpha}$, 
$\supp D$ is a set of non-smooth points of $f_{\alpha}$ for some $\alpha$.
We call that a divisor $D=\{  (U_{\alpha}, f_{\alpha}) \}_{\alpha}$ {\it passes} $p$ if $p\in \supp D$.

\begin{rmk}
For a divisor $\{  (U_{\alpha}, f_{\alpha})_{\alpha} \}$,
each $(U_{\alpha}, f_{\alpha})$ defines a weighted balanced polyhedral complex in $U_{\alpha}$
 of codimension one and we can glue these complexes.
That is, a Cartier divisor defines 
a weighted balanced polyhedral complex in $X$ of codimension one.
Conversely a weighted balanced polyhedral complex in $X$ of codimension one defines 
a Cartier divisor. 
Thus we can identify an equivalent class of a Cartier divisor and
a weighted balanced polyhedral complex in $X$ of codmension one.
\end{rmk}

\begin{dfn}
A divisor $D$ is {\it effective}
if there exists a representation $\{ (U_{\alpha}, f_{\alpha}) \}_{\alpha}$ of $D$ such that
all $f_{\alpha}$ are regular functions.
\end{dfn}

Two divisors 
$\{  (U_{\alpha}, f_{\alpha}) \}_{\alpha}$ and
$\{  (V_{\beta}, g_{\beta}) \}_{\beta}$
are {\it linearly equivalent} if there exists a rational function $h$ on $X$ such that
$f_{\alpha}=g_{\beta}+h$ on each $U_{\alpha}\cap V_{\beta}\neq \emptyset$.

For a divisor $D$,
we denote a set of effective divisors linearly equivalent to $D$ by $|D|$.

\begin{dfn}[\cite{LA}, Definitions 1.5]
Let $X$ be an $n$-dimensional tropical variety.
A {\it tropical line bundle} on $X$ is a tuple $(L, \pi ,\{ U_{i},\Psi_{i} \}_{i})$ of a topological space
$L$, a continuous surjection $\pi : L \twoheadrightarrow M$,
an open covering $\{ U_{i} \}$ called the trivializing covering
and homeomorphisms $\Psi_{i}:\pi^{-1}(U_{i})\cong U_{i}\times \T$ 
called trivializations which satisfy:
\begin{itemize}
\item  
The following diagram is commute:
\[\xymatrix{ \pi^{-1}(U_{i}) \ar[r]^{\Psi_{i}} \ar[rd]_{\pi} & U_{i}\times \T 
\ar[d] 
 \\ & U_{i} } .\]
Here the vertical map is the first projection;
\item For every $i,j$ with $U_{i}\cap U_{j}\neq \emptyset$,
there exists an invertible regular function  $\varphi_{ij} \colon U_{i}\cap U_{j}\to \R$ 
such that
$\Psi_{j}\circ\Psi_{i}^{-1}\colon (U_{i}\cap U_{j})\times \T\to (U_{i}\cap U_{j})\times \T$
is given by $(x,t)\mapsto (x,``\varphi_{ij}(x) t")$.
These $\varphi_{ij}$ are called a {\it transition functions}  of $L$.
\end{itemize}
It is clear that 
transition functions $\varphi _{ij}$ satisfies the cocycle condition 
$``\varphi_{ij}\cdot \varphi_{jk}"=\varphi_{ik}$.
We identify two tropical line bundles $(L, \pi ,\{ U_{i},\Psi_{i} \}_{i\in I})$ and $(L, \pi ,\{ U_{j},\Psi_{j} \}_{j\in J})$
if $(L,\pi ,\{ U_{i},\Psi_{i} \}_{i\in I\sqcup J})$ is a tropical line bundle on $M$, that is, 
satisfies the second condition above.
We often write $L$ for $(L, \pi ,\{ U_{i},\Psi_{i} \}_{i})$ to avoid heavy notation.
\end{dfn}

\begin{dfn} 
We take two line bundles 
$(L_{1},\pi _{1}, \{ U_{1i},\Psi_{1i} \})$ , $(L_{2},\pi _{2}, \{ U_{2j},\Psi_{2j} \})$.
Let $\varphi^{1}_{ij}$, $\varphi^{2}_{ij}$ be the correspondence transition functions.
The tensor product $L_{1}\otimes L_{2}$ is defined to be
the tropical line bundle on $X$ whose transition functions are 
$``\varphi^{1}_{ij}\varphi^{2}_{ij}"=\varphi^{1}_{ij}+\varphi^{2}_{ij}$.

The inverse $L_{1}^{-1}$ of the line bundle $L_{1}$ is defined to be the tropical line bundle on $X$
whose transition functions are 
$``(\varphi^{1}_{ij})^{-1}" =-\varphi^{1}_{ij}$.
\end{dfn}

\begin{dfn} 
Let $\pi\colon L\to X$ be a line bundle on a tropical variety $X$ and 
$U$ be an open subset of $X$.
A function $s\colon U\to \pi^{-1}(U)$ is a {\it regular section} (resp. {\it rational section}) on $U$ of $L$
if $\pi\circ s$ is the identity map on $U$ and
$p_{i}\circ \Psi_{i}\circ s$ is a regular (resp. rational) function,
where $p_{i}\colon U_{i}\times \T \to \T$ is the second projection.
If $U=M$ then we call $s$ a global regular (resp. rational) section 
or simply a regular (resp. rational) section.
The set $\Gamma (U,L)$ of regular sections on $U$ of $L$ 
naturally has the structure of a tropical module.
\end{dfn}
A rational section $s$ defines a divisor $\divisor s=\{  (U_{i}, p_{i}\circ\Psi_{i}\circ s) \}_{i}$.
Conversely, a divisor $D=\{  (U_{\alpha}, f_{\alpha}) \}_{\alpha}$ defines a line bundle $\mathcal{O}(D)$.
This correspondence is compatible for a sum of divisors and a tensor product of line bundles.

Assume $\Gamma (X,\mathcal{O}(D))\neq \emptyset$ and
fix a regular section  $s\in \Gamma (X,\mathcal{O}(D))$.
For any regular section $s'$, we can define a rational function $f=``\frac{s'}{s}"$.
Since $s$ and $s'$ are regular, $\divisor (f)+D\geq 0$.
Then we get a map $\Gamma (X,\mathcal{O}(D))\to \{ f\in \mathcal{M}(X)\mid \divisor (f)+D\geq 0 \}$.
It is clear that this map is
an isomorphism of $\T$-modules. Thus we get the following proposition.

\begin{prop}\label{section}
\[ \Gamma (X,\mathcal{O}(D))\cong \{ f\in \mathcal{M}(X)\mid \divisor (f)+D\geq 0 \}. \]
\end{prop}

At the end of this section, we note the following proposition.
\begin{prop}\label{gen}[\cite{HMY}, Corollary 8]
If a submodule $M$ of
the tropical module $\Gamma( X,\mathcal{O}(D))$ is finitely generated, 
this is generated by extremals.
This  generating set is minimal and unique up to tropical scalar multiplication.
\end{prop}
Precisely this proposition is a generalization of \cite{HMY}, Corollary 8,
and proved as the same way.

\section{Rank of divisor}

We consider a tropical variety $X$ with an atlas $\mathcal{V}=\{ V_{\beta} ,\Phi_{\beta} \}$
and a Cartier divisor $D$ on $X$. 

\begin{dfn}[Cartwright \cite{DC2}, Definition 3.1]
Let $X$ be a tropical variety and $D$ be a divisor on $X$.
We define $h^{0}(X,D)$ as
\[ h^{0}(X,D):=\min\left\{ k\in \Z_{\geq 0} \middle| 
\begin{aligned}
&\text{there exist $k$ points } p_{1},\ldots ,p_{k}, \text{ such that } \\
~&\hspace{0mm}\text{there is no divisor }
E\in |D| \text{ which passes all }p_{i}
\end{aligned}
\right\} .\]
If there is no such $k$, then we define $h^{0}(X,D)$ as $\infty$.
\end{dfn}
The value of $h^{0}(X,D)$ is difficult to compute in general.
In order to estimate this, we define two amounts $h_{a}^{0}(X,D)$ and $h_{b}^{0}(X,D)$.

Let $M$ be a finitely generated submodule of $\Gamma( X,\mathcal{O}(D))$ and
$s_{1},\ldots ,s_{k}$ be a minimal generating set of $M$.
By Proposition \ref{gen}, this minimal generating set is unique up to tropical multiplications.
Thus we can define
$U(M)$ to be the complement of
$\cup _{i=1}^{k}\supp (\divisor s_{i})$
and let $\{ U(M)_{\alpha}\}$ to be the set of connected components of $U(M)$.
 
Then restriction of generators
$s_{i}|_{U(M)_{\alpha}\cap V_{\beta}}$ are affine linear functions 
on each open subset $U(M)_{\alpha}\cap V_{\beta}$.
Thus we can assume that
$M|_{U(M)_{\alpha}\cap V_{\beta}}$
is a tropical submodule of
$\T [ x_{1},\ldots ,x_{n} ]$.
Here $x_{1},\ldots ,x_{n}$ are the coordinates of the ambient space of $\image \varphi_{\beta}$.

For $y\in U(M)_{\alpha}\cap V_{\beta}$,
we define $l_{M}(y)$ to be the number of monomials in $M|_{U(M)_{\alpha}\cap V_{\beta}}$ up to tropical multiplications.
It is clear that this number is independent for the choice of $\alpha, \beta$.
For $x\notin U(M)$, we define $U(M)_{x}$ to be
the union of $U_{\alpha}$ such that its closure contains $x$,
and define
\[ l_{M}(x):=\max_{y\in U_{x}}l_{M}(y) \]


\begin{dfn}\label{hahb}
We define two numbers $h^{0}_{a}(X,D)$ and $h^{0}_{b}(X,D)$ as the following;
\[ h^{0}_{a}(X,D):=
\max\left\{ k\in \Z_{\geq 0} \middle| 
\begin{aligned}
&\exists \text{ $k$ sections } s_{1},\ldots ,s_{k}, \text{ s.t. } \\
~&\hspace{0mm}\forall x\in X, l_{< s_{1},\ldots ,s_{k} >}(x)=k
\end{aligned}
\right\} \]

\begin{eqnarray*}
 h^{0}_{b}(X,D)&:=& \max\left\{ k\in \Z_{\geq 0} \middle| 
\begin{aligned}
&\forall x\in X, \exists \text{ $k$ sections } s_{1},\ldots ,s_{k}, \\
~&\hspace{0mm}\text{ s.t. } l_{< s_{1},\ldots ,s_{k} >}(x)=k
\end{aligned}
\right\} \\
&=&\min_{x\in X}\{ l_{\Gamma (X,\mathcal{O}(D))(x) } \}
\end{eqnarray*}
\end{dfn}

\begin{rmk}
These numbers are more easy to compute then $h^{0}(X,D)$
when we have a concrete expression of
a minimal generating set of $\Gamma (X,\mathcal{O}(D))$.
Moreover, these coincide the cardinal of a minimal generating set of $\Gamma (X,\mathcal{O}(D))$
 when the generators of $\Gamma (X,\mathcal{O}(D))$ have different slopes
each other around any point $x\in X$.
\end{rmk}

\begin{prop}[Main result 1]\label{Mainresult1}
\[ h^{0}_{a}(X,D)\leq h^{0}(X,D)\leq h^{0}_{b}(X,D)\]
\end{prop}
\begin{proof}
When $\Gamma (X,\mathcal{O}(D))=\{ -\infty \}$ or 
this is generated only one element, 
it is clearly that $h^{0}(X,D)=h^{0}_{a}(X,D)=h^{0}_{b}(X,D)$.
Thus we assume that $\Gamma (X,\mathcal{O}(D))\neq \{ -\infty\}$ and 
$\Gamma (X,\mathcal{O}(D))$ is generated by $2$ or more elements. 


Let $l:=h^{0}_{a}(X,D)$. By definition of $h^{0}_{a}(X,D)$, we can take $l$ sections
$s_{1},\ldots, s_{l}$ satisfying that
$l_{<s_{1},\ldots, s_{l}>}(x)=l$ for all $x\in X$.
In order to show $h^{0}(X,D)\geq l$, we take arbitrary $l-1$-points $p_{1},\ldots ,p_{l-1}\in X$.

Now we define a regular section of $\mathcal{O}(D)$ by the Vandermonde determinant
\begin{align*}
V_{D}(x; p_{1},\ldots ,p_{l-1})&:=
\operatorname{det}_{\text{trop}}
\left( \begin{array}{ccc} 
s_{1}(x)&\cdots &s_{l}(x)\\ 
s_{1}(p_{1})&\cdots &s_{l}(p_{1})\\
\vdots&\ddots&\vdots\\
s_{1}(p_{l-1})&\cdots &s_{l}(p_{l-1})
\end{array}\right)\\
&=\max_{\sigma\in \mathfrak{S}_{l}}\left\{ s_{\sigma (1)}(x)+
\sum_{i=2}^{l}s_{\sigma (i)}(p_{i-1}) \right\}\\
&=\max_{i\in \{ 1,\ldots ,l\}  }\left\{ s_{i}(x)+
\max_{\sigma\in \mathfrak{S}_{l}, \sigma(1)=i }\sum_{j=2}^{l}s_{\sigma (j)}(p_{j-1}) \right\}.
\end{align*}


\begin{lem}
 $\operatorname{div }V_{D}(x; p_{1},\ldots ,p_{l-1})$ passes through the $l-1$ points $p_{1},\ldots, p_{l-1}$.
\end{lem}
\begin{proof}
Let us show that $\operatorname{div }V_{D}(x; p_{1},\ldots ,p_{l-1})$ passes through $p_{1}$.
It is clear when the value $\operatorname{div }V_{D}(p_{1}; p_{1},\ldots ,p_{l-1})$ is infinite,
thus we assume that $\operatorname{div }V_{D}(p_{1}; p_{1},\ldots ,p_{l-1})$ is finite.

Suppose that the maximum in $\operatorname{div }V_{D}(p_{1}; p_{1},\ldots ,p_{l-1})$ is realized by
just one element $i_{max}$. We can assume this $i_{max}$ is $1$.
Let $\sigma_{1}$ be an element satisfying that 
$\max_{\sigma\in \mathfrak{S}_{l}, \sigma(1)=1 }\sum_{i=2}^{l}s_{\sigma (j)}(p_{j-1})
=\sum_{i=2}^{l}s_{\sigma_{i} (1)}(p_{j-1})$, that is,
$\operatorname{div }V_{D}(p_{1}; p_{1},\ldots ,p_{l-1})=
s_{1}(p_{1})+s_{\sigma_{1} (1)}(p_{1})+s_{\sigma_{1} (2)}(p_{2})+\cdots +s_{\sigma_{1} (l-1)}(p_{l-1})$.
Then the following equality holds;
\begin{align*}
V_{D}(p_{1}; p_{1},\ldots ,p_{l-1})&=
s_{1}(p_{1})+\max_{\sigma\in \mathfrak{S}_{l}, \sigma(1)=1 }\sum_{j=2}^{l}s_{\sigma (j)}(p_{j-1})\\
&=s_{\sigma_{1}(2)}(p_{1})+
\max_{\sigma\in \mathfrak{S}_{l}, \sigma(1)=\sigma_{1}(2) }
\sum_{j=2}^{l}s_{\sigma (j)}(p_{j-1}).
\end{align*}
That is, the maximum in $\operatorname{div }V_{D}(p_{1}; p_{1},\ldots ,p_{l-1})$ is realized by
two elements $i=1$ and $i=\sigma_{1}(2)$. This is contradiction.

Thus the maximum in $\operatorname{div }V_{D}(p_{1}; p_{1},\ldots ,p_{l-1})$ is realized by
two elements $i_{1},i_{2}$.
Then $\operatorname{div }V_{D}(x; p_{1},\ldots ,p_{l-1})$ passes $p_{1}$ 
since the slopes of $s_{i_{1}}$ and $s_{i_{2}}$ are distinct by definition of $s_{i}$'s.
\end{proof}

By the above lemma, we get $h^{0}(X,D)\geq l$.

Next, we show that $h^{0}(X,D)\leq h^{0}_{b}(X,D)$.
Let $L:=h^{0}_{b}(X,D)$.
By definition of $h^{0}_{b}(X,D)$ and $l_{\Gamma (X,\mathcal{O}(D)) }(x)$, 
there is a point $x'\in U(\Gamma (X,\mathcal{O}(D)))$
 satisfying that $l_{\Gamma (X,\mathcal{O}(D)) }(x')=L$.

Suppose $h^{0}(X,D)\geq L+1$ and take general $L$ points $p_{1},\ldots ,p_{L}$ near $x'$.
That is, if $x'\in U(\Gamma (X,\mathcal{O}(D)))_{\alpha}$, we take $L$ points in
the same component $U(\Gamma (X,\mathcal{O}(D)))_{\alpha}$.
Then 
we can take $L$ sections $s_{1},\ldots ,s_{L}$ and a global section
$s := \max_{1\leq i\leq L}\{  s_{i}(x)+t_{i}\}$ passing through all $p_{i}$. 

Now we define a multigraph $\Gamma$ whose vertex set is $\{ 1,\ldots ,L \}$.
For each $k=1,\ldots ,L$, we can choose two distinct elements $i_{k}\neq j_{k}\in \{ 1,\ldots ,L\}$
such that $s_{i_{k}}(p_{k})+t_{i_{k}}=s_{j_{k}}(p_{k})+t_{j_{k}}$.
We connect $i_{k}$ and $j_{k}$ by an edge for each $k=1,\ldots ,L$.
Then this multigraph $\Gamma$ has $L$ vertices and $L$ edges, thus it has a cycle.


Assume $1,\ldots ,k$ make a cycle. Then we get
\begin{align*}
s_{1}(p_{i_{1}})+t_{1}&=s_{2}(p_{i_{1}})+t_{2}\\
s_{2}(p_{i_{2}})+t_{2}&=s_{3}(p_{i_{2}})+t_{3}\\
~&\vdots \\
s_{k-1}(p_{i_{k-1}})+t_{k-1}&=s_{k}(p_{i_{k-1}})+t_{k}\\
s_{k}(p_{i_{k}})+t_{k}&=s_{1}(p_{i_{k}})+t_{1}
\end{align*}
for some $p_{i_{1}},\ldots p_{i_{k}}$. By combining these equations, we get
\[ \left( s_{1}(p_{i_{1}})-s_{2}(p_{i_{1}})\right)
+\cdots 
+\left( s_{k}(p_{i_{k}})-s_{1}(p_{i_{k}})\right)=0. \]

More generally, we define the functions $l_{i_{1},\ldots ,i_{k}}$ for $k>1$ and
 a subset $\{ i_{1},\ldots ,i_{k}\}\subset \{ 1,\ldots ,L \}$ as follows:
\[ l_{i_{1},\ldots ,i_{k}}(x_{1},\ldots x_{k})
:=\sum_{j=1}^{k}\left(
s_{i_{j}}(x_{j})-s_{i_{j+1}}(x_{j})
\right),
 \]
where $i_{k+1}:=i_{1}$.
These functions are rational functions on $X^{k}$.
To summarize the above argument, we obtain the following lemma.
\begin{lem}\label{l0}
Suppose $h^{0}(X,D)\geq L+1$. \\
Then for general $L$ points $p_{1}\ldots ,p_{L}$ near $x'$,
there exist ordered sequences $\{ i_{1},\ldots ,i_{k}\}, \{ j_{1},\ldots ,j_{k}\}\subset  \{ 1,\ldots ,L \}$
such that
$ l_{i_{1},\ldots ,i_{k}}(p_{j_{1}},\ldots p_{j_{k}})=0 $.
\end{lem}
The equations $l_{i_{1},\ldots ,i_{k}}(x_{j_{1}},\ldots x_{j_{k}})=0$
define a codimension $1$ topological subspace of $X^{L}$.
Thus
when we take sufficiently general $L$ points $p_{1},\ldots ,p_{L}$ near $x'$,
for any ordered sequences $\{ i_{1},\ldots ,i_{k}\}$ and $\{ j_{1},\ldots ,j_{k}\}$ of $\{ 1,\ldots ,L \}$,
$ l_{i_{1},\ldots ,i_{k}}(p_{j_{1}},\ldots p_{j_{k}})\neq 0$.\\
This contradicts to the above Lemma.
Thus we get $h^{0}(X,D)< L+1$.

We define subsets $H_{1}$, $H_{k}\subset (X)^{L}$ , $k=2,\ldots ,L$ as below.
\[ 
H_{1}:=\left\{ ( x_{1},\ldots ,x_{L} )\in (X)^{L} \middle| x_{i} \notin U(\Gamma(X,\mathcal{O}(D)))
 \text{ for some }i \right\}.
\]
\[
H_{k}:=\left\{ ( x_{1},\ldots ,x_{L} )\in (X)^{L} \middle| 
\begin{aligned}
\exists \text{an ordered sequence }\{ i_{1},\ldots ,i_{k}\}\text{ of elements in } \{ 1,\ldots ,L \}\\
~&\hspace{-87mm}\text{and a subset }\{ x_{j_{1}}, \ldots ,x_{j_{k}} \}\subset \{ x_{1},\ldots ,x_{L} \}\\ 
~&\hspace{-87mm} \text{satisfying } l_{i_{1},\ldots ,i_{k}}(x_{j_{1}},\ldots x_{j_{k}})=0
\end{aligned}\ \ \ \ \ \ \ \ \ \ \ \ \ \ \ \ 
\right\}.
\]

It is clear that the dimension of $H_{1}$ is $Ln-1$.
The dimension of $H_{k}$,for $k=2,\ldots ,l$ is also $Ln-1$
since functions $l_{i_{1},\ldots ,i_{k}}$ are rational functions on $X^{L}$.
Thus $X^{L}\setminus \cup_{k=1}^{L}H_{k}$ is nonempty.
That is, for general $L$ points $p_{1},\ldots ,p_{L}$ in $X^{L}\setminus \cup_{k=1}^{L}H_{k}$,
$ l_{i_{1},\ldots ,i_{k}}(p_{j_{1}},\ldots p_{j_{k}})\neq 0$ 
for any subset $\{ p_{j_{1}},\ldots p_{j_{k}} \}$ of these $L$ points 
and any ordered sequence $\{ i_{1},\ldots ,i_{k}\}$ of $\{ 1,\ldots ,L \}$.
This contradicts Lemma \ref{l0}.
Thus we get $h^{0}(X,D)< L+1$. 
\end{proof}

\begin{rmk}
In tropical geometry, the Identity theorem does not hold.
this facts seems to be related to whether or not
$h^{0}(X, D)$ is equal to $\dim \Gamma(X,\mathcal{O}(D))$.
\end{rmk}





\section{Tropical toric varieties}
Let $N\cong \Z^n$ be a lattice, $M:=N^{\vee}=\operatorname{hom}(N,\Z)$ be the dual lattice of $N$, 
and $N_{\R}:=N\otimes \R$, $M_{\R}:=M\otimes \R$ be the correspondence real vector spaces.

\begin{dfn}
Let $V=\Z^{n}\otimes \R$ be a real vector space with a lattice $\Z^{n}$.
A {\it rational polyhedral cone} in V is a cone in $V$ defined by
an intersection of finite half planes
$\left\{  x\in V \mid a_{i}(x)\geq 0 \right\}$ for some $a_{i}\in (\Z^{n})^{\vee}$.

A {\it face of $\sigma$} is a subset $\tau\subset \sigma$ such that
$\psi_{i} (\tau \cap U_{i})$ is an intersection of $\operatorname{Im}\psi_{i}$
and a face of $\rho_{i}$ in $\R^{n}$ for each $i$.
It is denoted by $\tau\prec \sigma$.

A {\it fan} $\Sigma$ in $V$ is a finite collection of 
rational polyhedral cones in $V$ satisfying the following:
\begin{itemize}
\item Any face of $\sigma \in \Sigma$ belongs to $\Sigma$;
\item For $\sigma ,\tau\in \Sigma$ with $\sigma\cap \tau\neq \emptyset$,
the intersection is a face of both $\sigma$ and $\tau$. 
\end{itemize}
We denote $\Sigma^{[k]}$ the set of cones in $\Sigma$ of dimension $k$.

A cone $\sigma$ is said to be {\it smooth} if it is generated by
a subset of a basis of $\Z^{n}$.
A fan $\Sigma$ is said to be {\it smooth} if every cone of $\Sigma$ is smooth.

A {\it dual cone} $\sigma^{\vee}$ of a cone $\sigma$ is a cone defined by
$\left\{ a\in V^{\vee}\mid  a(x)\geq 0\ \forall x\in \sigma  \right\}$.
If $\sigma$ is a rational polyhedral cone, then $\sigma^{\vee}$ is the same.
We define $S_{\sigma}$ to be $\sigma^{\vee}\cap M$.
\end{dfn}

\begin{dfn}[\cite{HM}, Definition 1.12]
Let $S\subset \Z^{n}$ be a semigroup generating by finite sets $\{ g_{1},\ldots ,g_{k} \}$
and $R\subset \Z^{k}$ be a generating set of $\{  r\in \Z^{k}\mid  \sum r_{i}g_{i}=0  \}$.
We define
\[ K(S):=\{  x\in \T^{k}\mid  x\cdot r^{+} = x\cdot r^{-}"\ \forall r\in R \} \]
where 
$r^{+}:=(\max (r_{1},0),\ldots , \max (r_{k},0))$ and
$r^{-}:=(\max (-r_{1},0),\ldots , \max (-r_{k},0))$ for $r=(r_{1},\ldots ,r_{k})$.
\end{dfn}

By Lemma 1.13 of \cite{HM},
a set $U_{\sigma}:=\operatorname{Hom}(S_{\sigma}, \T)$ is identified with 
$K(S_{\sigma })\subset \T^{k}$ naturally
and this identification gives $\operatorname{Hom}(S_{\sigma}, \T)$ a structure of tropical variety.
We call this tropical variety $U_{\sigma}$ a {\it tropical affine toric variety}.

Let $\tau$ be a face of a cone $\sigma$.
Then the inclusion $\tau\hookrightarrow \sigma$ induces an inclusion 
$\iota_{\tau, \sigma}\colon U_{\tau}\hookrightarrow U_{\sigma}$.

\begin{dfn}%
Let $\Sigma$ be a fan on $N_{\R}$.
A {\it tropical toric variety} $X_{\Sigma}$ is a tropical variety $(\coprod U_{\sigma})/\sim$,
where $x\sim y$ means that there exist three cones $\sigma_{1}, \sigma_{2}, \tau$
and $z\in U_{\tau}$
such that $\tau\subset \sigma_{1}\cap \sigma_{2}$,
$x\in U_{\sigma}$, $y\in U_{\sigma '}$, $\iota_{\tau, \sigma}(z)=x$ and
 $\iota_{\tau, \sigma'}(z)=y$.
A tropical toric variety $X_{\Sigma}$ is {\it smooth} if each cone of $\Sigma$ is smooth. 
\end{dfn}

In the following, we consider only smooth tropical toric varieties.
\begin{dfn}
Let $\sigma$ be a smooth cone in $N_{\R}$ and $\{ \rho_{1},\ldots ,\rho_{k}\}$ be
the set of rays of $\sigma$.
Let $e_{\rho_{i}}\in N$ be a primitive generator of $\rho_{i}$.
We fix a dual frame $\{ m_{\rho_{1}},\ldots ,m_{\rho_{k}}\} \subset M$
 of $\{ e_{\rho_{1}},\ldots ,e_{\rho_{k}} \}$, that is,
$m_{\rho_{i}}\in M$ be an element $m\in M$ satisfying that
$m.\cdot e_{\rho_{i}}=1$ and $m.\cdot e_{\rho_{j}}=0$ for $j\neq i$.
Then we define a regular function $\varphi_{\rho_{i},\sigma}$ on $U_{\sigma}$ by
$\varphi_{\rho_{i},\sigma}(x):=x(m_{\rho_{i}})$ for 
$x\in U_{\sigma}=\operatorname{Hom}(\sigma^{\vee}\cap M, \T)$.
\end{dfn}

Let $\Sigma$ be a smooth fan, $\rho$ be a ray of $\Sigma$, 
$\sigma_{1},\ldots ,\sigma_{l}$ be cones of $\Sigma$ with $\rho\prec \sigma_{i}$.
Then on $U_{\sigma_{i}}\cap U_{\sigma_{j}}$ for $i\neq j$, 
the difference of regular functions $\varphi_{\rho, \sigma_{i}}$ and $\varphi_{\rho, \sigma_{j}}$ are 
an invertible regular function on $U_{\sigma_{i}}\cap U_{\sigma_{j}}$.
Thus we define a {\it ray divisor} $D_{\rho_{i}}$ on $X_{\Sigma}$
to be a divisor represented by $( U_{\sigma}, \varphi_{\rho_{i}, \sigma} )$
and a {\it toric divisor} to be a 
divisor of the form $\sum_{\rho\in \Sigma^{1}}a_{\rho}D_{\rho}$, $a_{\rho}\in \Z$.

\begin{dfn}
Let $\Sigma$ be a smooth fan with rays $\{ \rho_{1},\ldots , \rho_{d} \}$.
The {\it canonical divisor} $K_{X_{\Sigma}}$ on $X_{\Sigma}$ is a divisor $\displaystyle -\sum_{i=1}^{d}D_{\rho_{i}}$.
\end{dfn}

\begin{rmk}
Mikhalkin\cite{M} defines the canonical class for general tropical varieties.
The above definition is the restriction of Mikhalkin's definition for tropical toric varieties.
\end{rmk}

The following lemma allows us to consider only toric divisors.
\begin{lem}[\cite{HM}, Lemma 2.42]\label{toricdiv}
Any Cartier divisor on a smooth tropical toric variety is linearly equivalent to a  toric divisor.
\end{lem}

Let $D=\sum_{\rho}a_{\rho}D_{\rho}$ be a toric divisor of a tropical toric variety $X_{\Sigma}$.
Then Proposition \ref{section} said
$\Gamma (X,\mathcal{O}(D))\cong \{ f\in \mathcal{M}(X)\mid \divisor (f)+D\geq 0 \}$.
Let $f$ be a rational function on $X$ with $\divisor (f)+D\geq 0$.
Then $f|_{U_{0}}$ is a regular function on $U_{0}=N_{\R}$, thus we can consider that
$f$ is written as  $``\sum_{m\in A}c_{m}x^{m}"$ for some finite subset $A\subset M$.

In order to describe $\Gamma(X,\mathcal{O}(D))$ by its generators, we show the following lemma.

\begin{lem}[]
\[ \divisor (x^{m})=\sum_{\rho}\langle m,e_{\rho} \rangle D_{\rho}\]
\end{lem}
\begin{proof}
It is enough to show the above equation on $U_{\sigma}$ for a maximal cone $\sigma$.
Let $\rho_{1},\ldots ,\rho_{k}$ be rays of $\sigma$, $e_{\rho_{i}}$ be the primitive generator of $\rho_{i}$,
and $\{ m_{\rho_{1}},\ldots ,m_{\rho_{k}}\}$ be the dual basis of $\{ e_{\rho_{1}},\ldots ,e_{\rho_{k}} \}$.
Then $m\in M$ can be written as $\sum \langle m_{\rho_{i}},e_{\rho_{i}}\rangle m_{\rho_{i}}$.
Thus $\divisor (x^{m})=\sum \langle m_{\rho_{i}},e_{\rho_{i}}\rangle D_{\rho_{i}}$
by the definition of $\varphi_{\rho,\sigma}$.
\end{proof}

\begin{dfn}
Let $f=``\sum_{m\in A}c_{m}x^{m}"$ be a rational function on $X_{\Sigma}$.
Take a generic point $p$ of $\supp D_{\rho}$ and a neighborhood $V$ of $p$ such that
$f|_{V\cap N_{\R}}$ is an affine linear function $``c_{m}\cdot x^{m}"=c_{m}+\langle m,x \rangle$.
Then the {\it degree of $f$ along a ray divisor $D_{\rho}$}, denoted by $\deg (f)_{\rho}$,
 is defined as $\langle m, e_{\rho}\rangle$.
This is not depend on the choice of a generic point $p$
since the degree of $f$ along $D_{\rho}$ coincides with 
$\min_{m\in A}\{  \langle  m, e_{\rho} \rangle \mid  c_{m}\neq -\infty  \}$.
\end{dfn}

Then we get the following lemma.
\begin{lem}
$\divisor (f)= \divisor (f|_{\R^{n}})+\sum_{\rho\in \Sigma^{[1]}}\deg (f)_{\rho}D_{\rho}.$
\end{lem}

\begin{prop}\label{gens}
Let $D$ be a toric divisor on a smooth tropical toric variety $X$. Then
\[ \Gamma (X,\mathcal{O}(D))\cong ``\sum_{\divisor (x^{m})+D\geq 0 } \T \cdot x^{m}" . \]
\end{prop}
\begin{proof}
It is clear that
\[ \Gamma (X,\mathcal{O}(D))\cong \{ f\in \mathcal{M}(X)\mid \divisor (f)+D\geq 0 \} \supseteq
``\sum_{\divisor (x^{m})+D\geq 0 } \T \cdot x^{m}"\]
We show the inverse inclusion.
Let $f$ be a rational function on $X$ with $\divisor (f)+D\geq 0$.
This $f$ can be written as
$``\sum_{m\in A} c_{m}x^{m}"=\max_{m\in A}\{ m\cdot x +c_{m} \}$,
where $A$ is a finite subset of $M$ satisfying that
for all proper subset $A'\subset A$, $``\sum_{m\in A'} c_{m}x^{m}"\neq f$.

Suppose that there exists $m'\in A$ satisfying $\divisor (x^{m'})+D\ngeq 0$.
Then there exists a ray $\rho$ such that $\langle m', e_{\rho}\rangle +a_{\rho}<0$,
although $\deg (f)_{\rho}+a_{\rho}=\min_{m\in A} \langle m, e_{\rho}\rangle +a_{\rho}\leq
\langle m', e_{\rho}\rangle +a_{\rho}<0$.
This contradicts to $\divisor (f)+D\geq 0$.
\end{proof}


For a toric divisor $D=\sum a_{\rho}D_{\rho}$, we define a polyhedron $P(D)$ by
$P(D):=\{ m\in M_{\R} \mid \langle m, e_{\rho}\rangle +a_{\rho}\geq 0 \}$.
Then by Proposition \ref{gens}, we get $h^{0}_{a}(X,D) = h^{0}_{b}(X,D)=|P_{D}\cap M|$.
Thus we get the following result by Proposition \ref{Mainresult1}.

\begin{prop}\label{h0}
For a toric divisor  $D$ on a smooth tropical toric variety $X$,
\[h^{0}(X,D) = |P_{D}\cap M| . \]
\end{prop}
This result is compatible for the classical result.

\section{Intersection number of divisors on tropical toric surfaces}


By the intersection theory on tropical toric surface described in \cite{S0}, 
we get the following proposition.
\begin{prop}\label{Drho}
Let $\Sigma$ be a two dimensional smooth fan in $N_{\R}$.
Let $\rho_{1}$ and $\rho_{2}$ be two distinct rays of $\Sigma$.
Then the intersection number $D_{\rho_{1}}D_{\rho_{2}}$ of these is one if
there exists a cone $\sigma$ such that $\rho_{1}\prec \sigma$ and $\rho_{2}\prec \sigma$
and zero if else.

Let $\rho$ be a ray of $\Sigma$ with the primitive generating vector $u\in N$
and $\rho_{1}$, $\rho_{2}$ be distinct two rays of $\Sigma$ adjacent to $\rho$
with the primitive generating vectors $u_{1}, u_{2}\in N$.
Then the self-intersection number of $D_{\rho}$ is the number $b$ satisfying that
$u_{1}+u_{2}+bu=0$.
\end{prop}

\begin{rmk}
The latter of the above proposition is different to
Corollary 3.2.6 of Shaw \cite{S0} at a glance,
but these are equal.
\end{rmk}

\begin{thm}[Main result 2]
Let $\Sigma$ be a smooth fan in $N_{\R}$, $D$ be a divisor on $X_{\Sigma}$, 
$K_{X_{\Sigma}}$ be the canonical divisor of $X_{\Sigma}$ and 
$\chi (X_{\Sigma})$ be the topological Euler characteristic
of $X_{\Sigma}$.
Then 
\[ h^{0}(X_{\Sigma},D)+h^{0}(X_{\Sigma},K_{X_{\Sigma}}-D)\geq 
\chi (X_{\Sigma})+\frac{1}{2}D(D-K_{X_{\Sigma}}) \]
\end{thm}
\begin{proof}
By Lemma \ref{toricdiv}, we can assume that $D$ is a toric divisor $\sum_{\rho}a_{\rho}D_{\rho}$.

Let $X_{\Sigma}^{an}$ be a complex toric variety associated to the fan $\Sigma$,
$D_{\rho}^{an}$ be the classical ray divisor associated to a ray $\rho$,
$D^{an}$ be the classical divisor $\sum_{\rho}a_{\rho}D_{\rho}^{an}$ corresponding to $D$
and $K_{\Sigma}^{an}$ be a divisor $-\sum_{\rho}D_{\rho}^{an}$ corresponding to 
the torus-invariant canonical divisor $K_{\Sigma}=-\sum_{\rho}D_{\rho}$. of $X_{\Sigma}^{an}$.
Then classical Riemann-Roch inequality for compact smooth toric surfaces says that
\[ h^{0}(X_{\Sigma}^{an},D^{an})+h^{0}(X_{\Sigma}^{an},K_{X_{\Sigma}^{an}}-D^{an})\geq 
\chi (\mathcal{O}_{X_{\Sigma}^{an}})+\frac{1}{2}D^{an}(D^{an}-K_{X_{\Sigma}^{an}}) \]

Proposition \ref{h0} tells us the fact that the value of the right-hand side of the 
above inequality coincides with the values of corresponding classical one. That is,
\[
h^{0}(X_{\Sigma},D)+h^{0}(X_{\Sigma},K_{X_{\Sigma}}-D)=
h^{0}(X_{\Sigma}^{an},D^{an})+h^{0}(X_{\Sigma}^{an},K_{X_{\Sigma}^{an}}-D^{an}).
\]
On the other hand, by Proposition \ref{Drho} we get 
\[ 
\chi (X_{\Sigma})+\frac{1}{2}D(D-K_{X_{\Sigma}})=
\chi (\mathcal{O}_{X_{\Sigma}^{an}})+\frac{1}{2}D^{an}(D^{an}-K_{X_{\Sigma}^{an}}) .
\]
Here we use $\chi (X_{\Sigma})=\chi (\mathcal{O}_{X_{\Sigma}^{an}})=1$.
Thus we get the inequality by the classical inequality.
\end{proof}

\end{document}